\documentclass[11pt]{article}\textwidth 160mm\textheight 235mm
\usepackage{amsfonts} \usepackage{amssymb} \usepackage{graphicx}
\usepackage{amsmath} \usepackage{amsthm} \usepackage{mathrsfs}
\usepackage{tikz} \usepackage{cancel} \usepackage{todonotes}
\usetikzlibrary{matrix} \usetikzlibrary{arrows,shapes}
\usepackage{cite} \usepackage{hyperref}
\usepackage{amsmath}
\usepackage{amsthm}
\usepackage{amsfonts}
\usepackage{multirow}
\usepackage{amssymb}
\usepackage{graphicx}
\usepackage{tcolorbox}
\usepackage{booktabs}

\newtheorem{Theorem}{Theorem}[section]

\newtheorem{Lemma}[Theorem]{Lemma}
\newtheorem{Corollary}[Theorem]{Corollary}
\newtheorem{Definition}[Theorem]{Definition}

\usepackage{hyperref}
\hypersetup{colorlinks=true,urlcolor=blue}
\expandafter\let\expandafter\oldproof\csname\string\proof\endcsname
\let\oldendproof\endproof
\renewenvironment{proof}[1][\proofname]{
\oldproof[\ttfamily\scshape \bf #1.]
}{\oldendproof}

\DeclareMathOperator*{\argmin}{arg\,min}
\def\tilde{\widetilde}
\def\emp{\emptyset}  
\def\dom{{\rm dom}\,}  \def\epi{{\rm epi\,}}

\def\ox{\overline{x}} \def\oy{\overline{y}} 
  \def\disp{\displaystyle}
 
 \def\Bar{\overline}

\def\ra{\rangle}
\def\la{\langle} \def\e{\varepsilon}
\def\ox{\bar{x}} \def\oy{\bar{y}} 
 \def\ov{\bar{v}}

\def\gph{\mbox{\rm gph}\,} \def\epi{\mbox{\rm epi}\,}
 \def\dom{\mbox{\rm dom}\,}

  \def\O{\Omega}
\def\ph{\varphi} \def\emp{\emptyset} 
\def\oR{\Bar{\R}} \def\lm{\lambda}  \def\dd{\delta}
  \def\N{{\rm I\!N}}
\def\R{{\rm I\!R}}
\def\ve{\varepsilon}
  \def\e{\varepsilon} 
\def\Limsup{\mathop{{\rm Lim}\,{\rm sup}}}

\def\Limsup{\mathop{{\rm Lim}\,{\rm sup}}}
\def\Liminf{\mathop{{\rm Lim}\,{\rm inf}}}
\def\Lim{\mathop{{\rm Lim}\,{\rm }}}

\numberwithin{equation}{section}

\title{\bf
Variational Convexity of Functions in Banach Spaces}

\author{Pham Duy Khanh\footnote{Department of Mathematics, Ho Chi Minh City University of Education, Ho Chi Minh City, Vietnam. E-mail: pdkhanh182@gmail.com} \quad Vu Vinh Huy Khoa\footnote{Department of Mathematics, Wayne State University, Detroit, Michigan, USA. E-mail: khoavu@wayne.edu. Research of this author was partly supported by the US National Science Foundation under grants DMS-1808978 and DMS-2204519.}\quad Boris S. Mordukhovich\footnote{Department of Mathematics, Wayne State University, Detroit, Michigan, USA. E-mail: aa1086@wayne.edu. Research of this author was partly supported by the US National Science Foundation under grants DMS-1808978 and DMS-2204519, and by the Australian Research Council under Discovery Project DP-190100555.}\quad Vo Thanh Phat\footnote{Department of Mathematics, Wayne State University, Detroit, Michigan, USA. E-mail: phatvt@wayne.edu. Research of this author was partly supported by the US National Science Foundation under grants DMS-1808978 and DMS-2204519.}}
\begin{document}
\maketitle

\noindent
{\small{\bf Abstract}.  This paper addresses the study and characterizations of variational convexity of extended-real-valued functions on Banach spaces. This notion has been recently introduced by Rockafellar, and its importance has been already realized and applied to continuous optimization problems in finite-dimensional spaces. Variational convexity in infinite-dimensional spaces, which is studied here for the first time, is significantly more involved  and requires the usage of powerful tools of geometric functional analysis together with variational analysis and generalized differentiation in Banach spaces. \\[1ex]
}

\section{Introduction}\label{intro}
\vspace*{-0.1in}

The notion of {\em variational convexity} for extended-real-valued functions has been recently introduced and studied by Rockafellar \cite{rtr251} in finite-dimensional spaces. This notion is different from the standard local convexity of functions while offering instead a new insight on local behavior of the function in question via a certain local maximal monotonicity of its subdifferential. Such a novel viewpoint occurs to be useful in applications to various issues in continuous optimization including the design and justification of numerical algorithms; see \cite{roc,r22}. Useful subdifferential characterization of variational convexity are given in \cite{rtr251}. More recently, it has been revealed in 
\cite{kmp22convex} that the variational convexity of extended-real-valued function in finite-dimensional spaces is equivalent to the standard local convexity of their Moreau envelopes. Furthermore, paper \cite{kmp22convex} contains applications of the obtained characterization of variational convexity via Moreau envelopes to {\em variational sufficiency} in constrained nonsmooth optimization.

We are not familiar with any publications on variational convexity of functions defined on infinite-dimensional spaces, which is the main topic of this paper. The main goals here are to establish such characterizations in general frameworks of Banach spaces and extend in this way finite-dimensional characterization from both papers \cite{rtr251} and  \cite{kmp22convex}. As seen below, accomplishing these goals is a highly challenging task, which requires the usage and developments of delicate  advanced techniques of Banach space geometry and variational analysis in infinite dimensions with considering a variety of appropriate settings in Banach spaces.\vspace*{0.05in}

The subsequent material is organized as follows. Section~\ref{prel} collects the basic definitions,  preliminaries, and discussions from functional and variational analysis, which are broadly used in the paper. In Section~\ref{sec:localsub}, we define and discuss the basic notion of {\em variational convexity} of extended-real-valued functions on Banach spaces and establish its {\em graphical subdifferential characterization} in the case of reflexive spaces with the necessity part holding in general Banach spaces. Section~\ref{moreau} provides characterizations of variational convexity via {\em local maximal monotonicity} of  {subdifferential} and {\em local convexity} of {\em Moreau envelopes} of functions defined on {\em uniformly convex spaces}, where some implications hold in more general Banach space settings.\vspace*{0.05in}
 
Throughout of the paper, we use the standard notation and terminology of variational analysis and Banach space theory; see, e.g., \cite{FA,Mordukhovich06,phelps,Rockafellar98}. 

\section{Preliminaries and Initial Discussions}\label{prel}

Unless otherwise stated, the space $X$ in question is {\em Banach} with the norm $\|\cdot\|$ and its topological dual $X^*$, where $\langle \cdot, \cdot \rangle $ indicates the canonical pairing between $X$ and $X^*$. The symbol $\rightarrow$ refers to the strong/norm convergence, while $\overset{w}{\rightarrow}$ and $\overset{w^*}{\rightarrow}$ signify the weak and weak$^*$ convergence, respectively. Given a set-valued 
mapping/multifunction $F:X\rightrightarrows X^*$ from  a Banach space $X$ to it dual space $X^*$, the notation
\begin{eqnarray}\label{eq:P-K}
\Limsup\limits_{x \rightarrow \ox} F(x):=\big\{ x^*\in X^* \;\big|\; \exists \text{ seqs. } x_k \rightarrow \ox,\; x_k^*\overset{w^*}{\rightarrow} x^* \text{ with } x_k^* \in F(x_k),\;k\in\N\big\}
\end{eqnarray}
signifies the ({\em sequential}) {\em Painlev\'e-Kuratowski outer limit} of $F(x)$ as $x\rightarrow \ox$. Consider further the {\em duality mapping} $J:X\rightrightarrows X^*$ between $X$ and $X^*$ defined by
\begin{equation}\label{duality}
J(x):=\big\{x^* \in X^*\;\big|\;\la x^*,x\ra =\|x\|^2 =\|x^*\|^2\big\},\quad x\in X,
\end{equation}
which plays a significant role in what follows. The notation $\mathbb{B}_{r}(x)$ and $B_r (x)$, stands, respectively, for the closed ball and open ball centered at $x$ with radius $r>0$. \vspace*{0.05in}

Given an extended-real-valued function $\varphi\colon X\rightarrow\overline{\R}:=(-\infty,\infty]$, we always assume that $\varphi$ is {\em proper}, i.e., $\dom\ph:=\{x\in X \;|\;\varphi(x)<\infty\}\ne\emp$. For any $\ve\ge 0$, the {\em $\varepsilon$-subdifferential} of $\varphi$ at $\bar{x}\in \dom \varphi$ is defined by
\begin{equation}\label{FrechetSubdifferential}
\widehat\partial_\varepsilon \varphi(\ox):=\Big\{x^*\in X^*\;\Big|\;\liminf\limits_{x\to \overline x}\frac{\varphi(x)-\varphi(\ox)-\langle x^*,x-\ox\rangle}{\|x-\ox\|}\ge  - \varepsilon\Big\}.
\end{equation}
When $\varepsilon = 0$ in \eqref{FrechetSubdifferential}, this construction is called the (Fr\'echet) \textit{regular subdifferential}
of $\varphi$ at $\ox$ and is denoted by $\widehat{\partial}\varphi(\ox)$. The (Mordukhovich) {\em limiting/basic subdifferential} of $\varphi$ at $\ox$ is defined via the sequential outer limit \eqref{eq:P-K} by
\begin{equation}\label{MordukhovichSubdifferential}
\partial\varphi(\bar{x}):=\Limsup_{\substack{x\overset{\varphi}{\to}\bar{x} \\ \varepsilon \downarrow 0 }}\; \widehat{\partial}_\varepsilon\varphi(x),
\end{equation}
where $x\overset{\varphi}{\to}\bar{x}$ means that $x\rightarrow\bar{x}$ with $\varphi(x)\rightarrow\varphi(\bar{x})$. In the case where $\widehat{\partial}\varphi(\ox)=\partial\varphi(\ox)$, the function $\varphi$ is called {\em lower regular} at $\ox$; see \cite{Mordukhovich06}. It is well known that both \eqref{FrechetSubdifferential} and \eqref{MordukhovichSubdifferential} reduce to the classical gradient $\nabla\ph(\ox)$ for continuously differentiable (in fact strictly differentiable) functions. If $\varphi$ is convex, then both of these  subdifferentials reduce to the standard subdifferential of convex analysis defined as
\begin{equation*}
\partial \varphi (\ox) :=\big\{x^* \in X^*\;\big|\;\la x^*, x-\ox\ra \le \varphi (x) - \varphi (\ox)\;\text{ for all }\;x\in X\big\}. 
\end{equation*}
Thus the aforementioned two classes of functions exhibit lower/subdifferential regularity, while the latter collection of extended-real-valued functions is much broader; see, e.g., \cite{Mordukhovich06,Mor18,Rockafellar98}. 

It follows from \cite[Theorem~2.34]{Mordukhovich06} that we can equivalently put $\ve=0$ in \eqref{MordukhovichSubdifferential} if $\ph$ is lower semicontinuous (l.s.c.) around $\ox$ and the space $X$ is {\em Asplund}, i.e., a Banach space where each separable subspace has a separable dual. This class of spaces is fairly large including, e.g., every reflexive Banach space and every Banach space which dual is separable; see \cite{Borwein2005, fabian,Mordukhovich06,phelps} with more details and the references therein. It has been well recognized in variational analysis that the limiting subdifferential \eqref{MordukhovichSubdifferential} and the associated constructions for sets and set-valued mappings enjoy {\em full calculi} in Asplund spaces with a variety of applications presented in the two-volume book  \cite{Mordukhovich06}, while finite-dimensional specifications can be found in \cite{Mor18,Rockafellar98}. Some useful results for \eqref{MordukhovichSubdifferential} hold in general Banach spaces; see, e.g., \cite[Chapter~1]{Mordukhovich06}. \vspace*{0.05in}

The next material is mostly taken from \cite{Rockafellar98} in the case of finite dimensions and from \cite{bt05,Thibault04} where extensions of these notions are studied in infinite-dimensional spaces. An l.s.c.\ function $\varphi:X\to \overline{\R}$ on a Banach space $X$ is \textit{prox-bounded} if it majorizes a quadratic function, i.e., 
\begin{equation}\label{prox-bounded}
\varphi(x)\ge\alpha\|x-\ox\|^2+\beta\;\mbox{ for some }\;\alpha,\beta\in \R,\;\mbox{ and }
\;\ox\in X.
\end{equation}
An l.s.c.  function $\ph$ is {\em prox-regular} at $\bar{x}\in\dom\ph$ for $\bar{x}^*\in\partial\ph(\ox)$ if there exist numbers $\varepsilon>0$ and $r\ge 0$ such that we have the estimate
\begin{equation}\label{prox}
\varphi(x)\ge\varphi(u)+\langle u^*,x-u\rangle-\frac{r}{2}\|x-u\|^2
\end{equation}
for all $x\in\mathbb{B}_\varepsilon(\bar{x})$ and $(u,u^*)\in\gph\partial\varphi\cap (\mathbb{B}_\varepsilon(\ox)\times\mathbb{B}_\varepsilon(\ox^*))$ with $\varphi(u)<\varphi(\bar{x})+\varepsilon$. If \eqref{prox} holds for all $x^*\in\partial\varphi(\ox)$, then $\varphi$ is said to be {\em prox-regular at} $\ox$.   

Further, we say that $\varphi$ is {\em subdifferentially continuous} at $\bar{x}$ for $\ox^*\in \partial\varphi(\bar{x})$ if for any $\varepsilon>0$ there exists $\delta>0$ such that $|\varphi(x)-\varphi(\bar{x})|<\varepsilon$ whenever $(x,x^*)\in\gph\partial\varphi\cap (\mathbb{B}_\delta(\ox)\times\mathbb{B}_\delta(\ox^*))$. When this holds for all $x^*\in\partial\varphi(\ox)$, the function $\varphi$ is said to be {\em subdifferentially continuous at} $\ox$. It is easy to see that if $\varphi$ is subdifferentially continuous at $\bar{x}$ for $\bar{x}^*$, then the inequality  ``$\varphi(x)<\varphi(\ox)+\varepsilon$" in the definition of prox-regularity can be dropped. Extended-real-valued functions that are both prox-regular and subdifferentially continuous are called {\em continuously prox-regular}. This is a major class of extended-real-valued functions in second-order variational analysis that is a common roof for particular collections of functions important in applications as, e.g., the so-called amenable functions, etc.; see \cite[Chapter~13]{Rockafellar98} and \cite{bt05,Thibault04}. 

The \textit{$\varphi$-attentive $\varepsilon$-localization} of the subgradient mapping $\partial \varphi$ around $(\ox,\ox^*)\in \gph \partial \varphi$ is the set-valued mapping $T^{\varphi}_{\varepsilon}:X\rightrightarrows X^*$ defined by 
\begin{equation}\label{localization}
\gph T^{\varphi}_{\varepsilon} :=\big\{ (x,x^*)\in \gph \partial \varphi \;\big|\; \|x-\ox\|<\varepsilon,\; |\varphi(x)-\varphi(\ox)|<\varepsilon\;\text{ and }\;\|x^*-\ox^*\|<\varepsilon\big\}.
\end{equation}
If $\varphi$ is an l.s.c.\ function, a localization can be taken with just $\varphi(x) < \varphi(\ox) + \varepsilon$.\vspace*{0.05in}

Now we start exploring the notion of \textit{epi-convergence} of extended-real-valued functions on a Banach space $X$. More details and references can be found in \cite{Rockafellar98} in finite dimensions and in \cite{Attouch84,Borwein2005} in Banach spaces. Let $\N:=\{1,2,\ldots\}$, and let
\begin{eqnarray*}
\mathcal{N}_{\infty} &:=&\big\{N\subset\N\;\big|\;\N \setminus N \;\text{ finite}\big\},\\
\mathcal{N}_{\infty}^{\#} &:=& \big\{ N\subset \N\;\big|\;N\;\text{ infinite}\big\}.
\end{eqnarray*} 
For a sequence $\{C^k\}_{k\in\N}$ of subsets of $X$, the {\em outer limit} is the set 
\begin{eqnarray*}
\Limsup_{k\rightarrow \infty} C^k :&=&\big\{ x\in X\; \big| \; \exists N \in \mathcal{N}^{\#}_{\infty},\;\exists x^k \in C^k (k\in N)\;\text{ with }\;x^k \xrightarrow{N}x\big\} \\
&=& \big\{ x\in X\; \big| \; \forall V\in \mathcal{N}(x), \ \exists N \in \mathcal{N}_{\infty}^{\#},\ \forall k\in N: C^k \cap V \neq \varnothing\big\},
\end{eqnarray*}
while the {\em inner limit} of $\{C^k\}_{k\in\N}$ is defined by
\begin{eqnarray*}
\Liminf_{k\rightarrow \infty} C^k :&=& \big\{ x\in X\; \big| \; \exists N \in \mathcal{N}_{\infty},\ \exists x^k \in C^k (k\in N)\;\text{ with }\;x^k \xrightarrow{N}x\big\}\\
&=& \big\{ x\in X\; \big| \; \forall V\in \mathcal{N}(x), \ \exists N \in \mathcal{N}_{\infty},\ \forall k\in N: C^k \cap V \ne\emptyset\big\},
\end{eqnarray*} 
where $\mathcal{N}(x)$ denotes the collection of neighborhoods of $x$. 
The {\em limit} of the sequence $\{C^k\}$ exists if the outer and inner limit sets are equal, i.e.,
\begin{eqnarray*}
\Lim_{k\rightarrow \infty} C^k := \Limsup_{k\rightarrow \infty} C^k = \Liminf_{k\rightarrow \infty}C^k.
\end{eqnarray*}

We need the following simple observation. For a sequence of sets $E^k \subset X\times \R$ that are epigraphs of some extended-real-valued function, both the outer limit set $\Limsup_{k\rightarrow \infty} E^k$ and inner limit set $\Liminf_{k\rightarrow \infty} E^k$ are also epigraphs of some functions. Indeed, if either set contains $(x,\alpha)$, then it also contains $(x,\alpha^{\prime})$ for all $\alpha^{\prime} \in [\alpha,\infty)$. On the other hand, since both limit sets are closed, they intersect $\{x\}\times\R$ in a set which, unless empty, is a closed interval. Thus the criteria for being an epigraph of some function is satisfied. For any sequence $\left\{\ph^k\right\}_{k \in \mathbb{N}}$ of extended-real-valued functions on $X$, the \textit{lower epi-limit}, e-$\liminf_{k\rightarrow \infty}\ph^k$, is the function having as its epigraph the outer limit of the sequence of sets $\epi\ph^k$, i.e., 
$$
\epi (\text{e-}\liminf_{k\rightarrow \infty}\ph^k) := \Limsup_{k\rightarrow \infty} (\epi\ph^k).
$$
The \textit{upper epi-limit}, e-$\limsup_{k\rightarrow \infty}\ph^k$, is the function having as its epigraph the inner limit of the sets $\epi\ph^k$, which is defined by
$$
\epi (\text{e-}\limsup_{k\rightarrow \infty}\ph^k) := \Liminf_{k\rightarrow \infty} (\epi\ph^k).
$$
Thus $\text{e-}\liminf_{k\rightarrow \infty}\ph^k \le \text{e-}\limsup_{k\rightarrow \infty}\ph^k$ in general. When these two functions coincide, it is said that  the full limit \textit{epi-limit} function e-$\lim_{k\rightarrow \infty}\ph^k$ exists, i.e., we have
\begin{eqnarray*}
\text{e-}\lim_{k\rightarrow \infty}\ph^k := \text{e-}\liminf_{k\rightarrow \infty}\ph^k = \text{e-}\limsup_{k\rightarrow \infty}\ph^k.
\end{eqnarray*}
On this case, the functions $\ph^k$ \textit{epi-converge} to $\ph$, which is denoted by $\ph^k \xrightarrow{e}\ph$. It is clear that $\ph^k \xrightarrow{e}\ph\Longleftrightarrow \epi\ph^k \rightarrow \epi\ph$ as $k\to\infty$. 
The next characterization of epi-convergence taken from \cite[Proposition~1.14]
{Attouch84} is instrumental for deriving the main result of Section~\ref{sec:localsub}.

\begin{Lemma}\label{lem:epi-charac}
Let $X$ be a Banach space, and let $\varphi^k\colon X\to\oR$ as $k\in\N$ and 
$\varphi\colon X\to\oR$ be given functions. Then the following statements are equivalent:
\begin{enumerate}
\item[{\bf(i)}] $\varphi^k\xrightarrow{e}\varphi$ as $k\to\infty$.
\item[{\bf(ii)}] At each point $x\in X$, we have the relationship 
\begin{equation*}
\begin{cases}
\displaystyle\liminf_{k\to\infty}\varphi^k(x^k) \ge \varphi (x) \ \text{ for every sequence }\;x^k \rightarrow x,\\
\disp\limsup_{k\to\infty}\varphi^k (x^k) \le \varphi (x) \ \text{ for some sequence }\;x^k \rightarrow x.
\end{cases}
\end{equation*}
\end{enumerate}
\end{Lemma}
\vspace*{0.05in}

To proceed further, along with standard definitions of lower semicontinuity of functions in the strong/norm topology of $X$ and its {\em sequential} version, we recall their weak counterparts. A function $\varphi:X\rightarrow \overline{\R}$ is \textit{weakly sequentially lower semicontinuous} (weakly sequentially l.s.c.) at $\ox \in \dom \varphi$ if for any sequence $\{x_k\}$ which weakly converges to $\ox$, it holds that $\liminf_{k\rightarrow \infty}\varphi (x_k) \ge \varphi (\ox)$. We say that $\varphi$ is \textit{weakly lower semicontinuous} (weakly l.s.c.) at $\ox$ if for any $\varepsilon>0$ there exists a neighborhood $U$ of $\ox$ in the weak topology of $X$ such that
\begin{equation*}
\varphi (x)\ge \varphi (\ox)-\varepsilon\;\mbox{ for all }\;x\in U.
\end{equation*}
It is easily to see that the weak lower semicontinuity yields the weak sequential lower semicontinuity, while the reverse implication is not true in general; see, for instance, a counterexample in \cite[Example~2.1]{KYY15}. On the other hand, if $\varphi$ is weakly sequentially l.s.c.\ at $\ox$, then it is automatically l.s.c.\ at this point. The function $\varphi$ is called lower semicontinuous (resp. weakly sequentially lower semicontinuous, weakly lower semicontinuous) {\em around} $\ox$ if it possesses this property for all points in some neighborhood of $\ox$.\vspace*{0.05in}

Next we recall some notions related to monotonicity of set-valued mappings that are often called  operators in this framework. The classical monotonicity notion is formulated as follows. An operator $T:X\rightrightarrows X^*$ is (globally) {\em monotone} on $X$ if 
\begin{eqnarray*}
\la x_1^* - x_2^*, x_1- x_2 \ra \ge 0 \;\text{ whenever }\; (x_1,x_1^*),(x_2,x_2^*)\in \gph T.
\end{eqnarray*}
It is called to be (globally) \textit{strongly monotone} on $X$ with modulus $\kappa>0$ if 
\begin{eqnarray*}
\la x_1^* - x_2^*, x_1 - x_2 \ra \ge \kappa \|x_1-x_2\|^2 \;\text{ whenever }\; (x_1,x_1^*),(x_2,x_2^*)\in \gph T.
\end{eqnarray*}
A monotone (resp.\ strongly monotone) operator $T$ is \textit{maximal monotone} (resp.\ \textit{strongly maximal monotone}) if $\gph T = \gph S$ for any monotone operator $S:X\rightrightarrows X^*$ with $\gph T \subset \gph S$. We refer the reader to the monographs \cite{Bauschke2011,Borwein2005,phelps,Rockafellar98} for various properties and applications of monotone and maximal monotone operators in finite and infinite dimensions. Note, in particular, that the graph of any maximal monotone mapping is nonempty and closed.\vspace*{0.05in}

The {\em local monotonicity} of set-valued mappings is naturally defined as follows; see, e.g., \cite{Rockafellar98}.

\begin{Definition}\label{defi:local-mono} \rm 
Let $T:X\rightrightarrows X^*$, and let $(\ox,\ox^*)\in \gph T$. We say that:
\begin{itemize}\item[\textup{\textbf{(i)}}] 
$T$ is \textit{locally monotone} around $(\ox,\ox^*)$ if there are neighborhoods $U$ of $\ox$ and $V$ of $\ox^*$ with
\begin{eqnarray*}
\la x_1^*- x_2^*, x_1 - x_2 \ra \ge 0\;\text{ for all }\;(x_1,x_1^*),(x_2,x_2^*)\in \gph T \cap (U\times V).
\end{eqnarray*}
\item[\textup{\textbf{(ii)}}] 
$T$ is \textit{locally monotone} with respect to $W\subset X\times X^*$ if 
\begin{eqnarray*}
\la x_1^*- x_2^*, x_1 - x_2 \ra \ge 0\;\text{ for all }\;(x_1,x_1^*),(x_2,x_2^*)\in \gph T \cap W.
\end{eqnarray*}
\item[\textup{\textbf{(iii)}}] $T$ is \textit{locally strongly monotone} with modulus $\kappa>0$ around $(\ox,\oy)$ if there are neighborhoods $U$ of $\ox$ and $V$ of $\oy$ satisfying the estimate
\begin{eqnarray*}
\la x_1^* - x_2^*, x_1 - x_2 \ra \ge \kappa \|x_1-x_2\|^2\;\text{ for all }\;(x_1,x_1^*),(x_2,x_2^*)\in \gph T \cap (U\times V).
\end{eqnarray*}
\end{itemize}
\end{Definition}

It is easy to observe the following {\em robustness property}: if $T$ is locally monotone around $(\ox,\ox^*)$ with respect to a neighborhood $U\times V$, then we have the same property around any $(x,x^*)\in U\times V$. In this case, it is said that $T$ is locally monotone relative to $U\times V$. We now follow \cite{Pen02} to define the {\em local maximal monotonicity} of set-valued mappings.

\begin{Definition}\label{defi:local-max-P}\rm 
Let $T:X\rightrightarrows X^*$, and let $(\ox,\ox^*)\in \gph T$. Then we say that: 
\begin{itemize}
\item[\bf(i)] $T$ is \textit{locally maximal monotone} around $(\ox,\ox^*)$ if there exist a neighborhood $U\times V$ of $(\ox,\ox^*)$ and a maximal monotone operator $\overline{T}:X \rightrightarrows X^*$ such that
\begin{eqnarray*}
\gph \overline{T} \cap (U\times V) = \gph T \cap (U\times V).
\end{eqnarray*}
\item[\bf(ii)] $T$ is \textit{locally maximal monotone} with respect to $W\subset X\times X^*$ if there exists a maximal monotone operator $\overline{T}:X \rightrightarrows X^*$ such that
\begin{eqnarray*}
\gph \overline{T} \cap W = \gph T \cap W.
\end{eqnarray*}
\item[\bf(iii)] $T$ is \textit{locally strongly maximal monotone} around $(\ox,\ox^*)$ if there is a neighborhood $U\times V$ of $(\ox,\ox^*)$ and a strongly maximal monotone operator $\overline{T}: X \rightrightarrows X^*$ such that
\begin{eqnarray*}
\gph \overline{T} \cap (U\times V) = \gph T \cap (U\times V).
\end{eqnarray*}
\end{itemize}
\end{Definition}

Observe that if $T$ is locally maximal monotone with respect to $W\subset X\times X^*$, then it is also locally maximal monotone with respect to any subset $Q \subset W$.

\medskip

Finally in this section, we recall some geometric properties of Banach spaces; see, e.g., \cite{Ciora,fabian,nam} for more details. A norm $\|\cdot\|$ on a space $X$ is called (G\^ateaux) {\em smooth} if it is G\^ateaux differentiable on $X\setminus \{0\}$. Equivalently, a norm on $X$ is smooth if and only if the corresponding duality mapping $J$ defined in \eqref{duality} is single-valued on the whole $X$. Note that any separable Banach space admits an equivalent norm that is G\^ateaux differentiable off the origin.

The \textit{moduli of convexity and smoothness} of $X$ are defined, respectively, by
\begin{eqnarray*}
\xi_X (t):=\inf\bigg\{ 1 - \dfrac{\|x+y\|}{2} \;\bigg|\; x,y\in S_X,\; \|x-y\|=t\bigg\} \ \text{ for all }\;t\in [0,2],\quad{\rm and}
\end{eqnarray*}
\begin{eqnarray*}
\rho_X (s):= \sup\bigg\{\dfrac{1}{2}\Big(\|x+y\|-\|x-y\|\Big) -1\;\bigg|\; x\in S_X,\; \|y\|=s\bigg\} \ \text{ for } \;s>0,
\end{eqnarray*}
where $S_{X}:=\{x\in X\;|\;\|x\|=1\}$. 
A Banach space $X$ is said to be \textit{uniformly convex} if $\xi_X (t) >0$. It is called \textit{$2$-uniformly convex} (resp. \textit{$2$-uniformly smooth}) if there exists a constant $b>0$ such that $\xi_X (t) \ge bt^2$ (resp. $\rho_X (s) \le b s^2$). It is shown in \cite[Theorem~7~and~Theorem~8, respectively]{Bynum76} that any $2$-uniformly convex space (resp. $2$-uniformly smooth space) can be characterized via a lower (resp.\ upper) {\em weak  parallelogram law}:
\begin{eqnarray}\label{eq:LWP}
\|x+y\|^2 + c\|x-y\|^2 \le 2\big(\|x\|^2+\|y\|^2\big)\ \text{ for all }\;x,y\in X,
\end{eqnarray}
\begin{eqnarray}\label{eq:UWP}
\|x+y\|^2 + c\|x-y\|^2 \ge 2\big(\|x\|^2+\|y\|^2\big)\ \text{ for all }\;x,y\in X.
\end{eqnarray}
In particular, a Banach space $X$ is $2$-uniformly convex (resp. $2$-uniformly smooth) with constant $b>0$ if and only if it satisfies condition \eqref{eq:LWP} (resp.\ condition \eqref{eq:UWP}) for some $c_b >0$. Therefore, any Hilbert space is simultaneously a $2$-uniformly convex and $2$-uniformly smooth Banach space. Moreover, it is obvious that any $2$-uniformly convex space (resp. $2$-uniformly smooth space) is uniformly convex (resp.\ uniformly smooth), and hence it is reflexive by the Milman-Pettis theorem. Some well-known non-Hilbert $2$-uniformly convex (resp. $2$-uniformly smooth) spaces are $l^p (\mu)$ and $L^p(\mu)$ with $1<p<2$ (resp.\ with $2<p<\infty$). Recall \cite{brow} that in any reflexive Banach space, there exists an equivalent norm such that the corresponding duality mapping $J$ from \eqref{duality} is bijective and continuous with its inverse $J^{-1}$ also being continuous. If in addition $X$ is $2$-uniformly convex, then it follows from \cite{Xu91} that the duality mapping is strongly monotone.

\section{Subdifferential Characterization of Variational Convexity}\label{sec:localsub}

In this section, we formulate the underlying notion of {\em variational convexity} for extended-real-valued functions on Banach spaces and establish its characterization in terms of the limiting subdifferential \eqref{MordukhovichSubdifferential}. The following definition is taken from Rockafellar \cite{rtr251}, where it is formulated for functions on finite-dimensional spaces.

\begin{Definition}\label{def:vc}\rm  An l.s.c.\ function $\varphi:X \to \overline{\R}$ is called \textit{variationally convex} at $\bar{x}$ for $\bar{x}^* \in \partial \varphi(\bar{x})$ if for some convex neighborhood $U \times V$ of $(\bar{x}, \bar{x}^*)$ there exists a convex l.s.c.\ function $\widehat{\varphi} \leq \varphi$ on $U$ and $\ve>0$ such that we have the relationships
\begin{equation*}
\left(U_{\varepsilon} \times V\right) \cap \operatorname{gph} \partial \varphi=(U \times V) \cap \operatorname{gph} \partial \widehat{\varphi}\;\text{ and }\;\varphi(x)=\widehat{\varphi}(x)\;\text{ at the common elements }\;(x, x^*),
\end{equation*}
where $U_\ve:=\{x\in U\;|\;\ph(x)<\ph(\ox)+\ve\}$. 
\end{Definition}

As demonstrated and discussed in \cite{rtr251} and also in \cite{kmp22convex}, this notion is different from the usual local convexity of extended-real-valued functions.\vspace*{0.05in}

Now we present two lemmas on epi-convergence in minimization used in what follows. The first one is taken from \cite[Proposition~2.9]{Attouch84}.

\begin{Lemma}\label{Lem:2.9-Attouch}
Let $X$ be a Banach space, and let $\varphi^k\colon X\to\oR$ for $k\in \N$ and $\varphi\colon X\to\oR$ be l.s.c.\ functions such that $\varphi^k \xrightarrow{e}\varphi$ as $k\to\infty$. Then we have
\begin{equation}\label{Att-1}
\inf \varphi \ge \limsup_{k\to\infty} (\inf \varphi^k).   
\end{equation}
\end{Lemma}

The second rather simple lemma reveals behavior of argminimum sets under epi-convergence of extended-real-valued functions defined on Banach spaces.

\begin{Lemma}\label{Lem:fixed1}
Let $X$ be a Banach space, and let $\varphi^k\colon X\to\oR$ for $k\in \N$ and $\varphi\colon X\to\oR$ be l.s.c.\ functions. Suppose that $\varphi^k \xrightarrow{e}\varphi$ as $k\to\infty$ and that $\argmin \varphi^k \neq \emptyset$ for all $k\in \N$ as well as $\argmin \varphi \neq \emptyset$. Then we have the inclusion
\begin{equation*}
\Limsup_{k\to\infty} (\argmin \varphi^k) \subset \argmin \varphi.
\end{equation*}
Therefore, for any choice of $x^k \in \argmin \varphi^k$, $k\in\N$, the sequence $\{x^k\}$ has all its cluster points belonging to $\argmin \varphi$. If $\argmin \varphi$ consists of a unique point $\ox$, then $x^k \rightarrow \ox$ as $k\to\infty$.
\end{Lemma}
\begin{proof}
Pick $x \in\displaystyle\Limsup_{k\to\infty}\argmin \varphi^k$ and find subsequences $\{k_n\}$ of 
$\N$ and $x^{k_n}\in \argmin \varphi^{k_n}$ such that $x^{k_n}\rightarrow x$ as $n\rightarrow \infty$. Since $\varphi^k \xrightarrow{e}\varphi$, we have $\Limsup_{k\to\infty} (\epi \varphi^k) = \epi\varphi$. By 
\begin{equation*}
\big(x, \limsup_{n\to\infty}\varphi^{k_n} (x^{k_n})\big) = \lim_{l\rightarrow \infty}(x^{k_{n_l}}, \varphi^{k_{n_l}} (x^{k_{n_l}}))\;\text{ for some subsequence }\;\{k_{n_l}\}\; \text{ of }\;\{k_n\}
\end{equation*}
and $(x^{k_{n_l}}, \varphi^{k_{n_l}} (x^{k_{n_l}}))\in \epi (\varphi^{k_{n_l}})$ for all $l\in \N$, it follows that
\begin{equation*}
\big(x, \limsup_{n\to\infty} \varphi^{k_n} (x^{k_n})\big) \in \Limsup_{k\to\infty} (\epi \varphi^k).
\end{equation*}
Hence $\big(x,\limsup_{n\to\infty} \varphi^{k_n}(x^{k_n})\big)\in \epi \varphi$, i.e.,
\begin{equation*}
\varphi (x) \le \limsup_{n\to\infty} \varphi^{k_n} (x^{k_n}).
\end{equation*}
This readily implies that
\begin{equation*}
\varphi (x) \le \limsup_{n\to\infty} \varphi^{k_n} (x^{k_n}) = \limsup_{n\to\infty} (\inf \varphi^{k_n}) \le \limsup_{k\to\infty} (\inf \varphi^k) \le \inf \varphi,
\end{equation*}
where the last inequality is due to estimate \eqref{Att-1} in Lemma~\ref{Lem:2.9-Attouch}. This gives us $x \in \argmin \varphi$ and therefore completes the proof of this lemma.
\end{proof}

The next result provides a major characterization of variational convexity of extended-real-valued functions defined on Banach spaces. It is a far-going generalization of the finite-dimensional characterization given in \cite[Theorem~1,~(c)$\Longleftrightarrow$(b)]{rtr251}.

\begin{Theorem}\label{main:VC}
Let $\varphi:X \to \overline{\R}$ be an l.s.c.\ function defined on a Banach space $X$, and let $\ox^* \in \partial \varphi(\ox)$ be a basic subgradient from \eqref{MordukhovichSubdifferential}.  Consider the following assertions:
\begin{itemize} 
\item[\bf(i)] $\varphi$ is variationally convex at $\bar{x}$ for $\ox^*$.
\item[\bf(ii)] There exists a convex neighborhood $U \times V$ of $(\ox, \ox^*)$ along with $\varepsilon>0$ such that we have
\begin{equation}\label{ii}
\big[(u, u^*) \in\left(U_{\varepsilon} \times V\right) \cap \gph \partial \varphi\big] \Longrightarrow\big[\varphi(x) \geq \varphi(u)+ \la u^*,x-u\ra\big]\;\text{ for all }\ x \in U,
\end{equation}
where $U_\ve$ is taken from Definition~{\rm\ref{def:vc}}.
\end{itemize} 
Then implication {\bf(i)}$\Longrightarrow${\bf(ii)} holds in general Banach spaces. The reverse implication is satisfied if $X$ is a reflexive space and if $\varphi$ is weakly sequentially l.s.c.\ around $\ox$.
\end{Theorem}
\begin{proof} First we verify implication {\bf(i)}$\Longrightarrow${\bf(ii)} in any Banach space $X$. Assuming the variational convexity of $\varphi$ at $\ox$ for $\ox^*$, find a convex neighborhood $U\times V$ of $(\ox,\ox^*)$ and a convex l.s.c.\ function $\widehat{\varphi}\le \varphi$ on $U$ such that
\begin{equation*}
\left(U_{\varepsilon} \times V\right) \cap \operatorname{gph} \partial \varphi=(U \times V) \cap \operatorname{gph} \partial \widehat{\varphi} \text{ and } \varphi(x)=\widehat{\varphi}(x)\; \text{ at the common elements }\;(x, x^*)
\end{equation*}
for some $\ve>0$. Without loss of generality, suppose that $\widehat{\varphi}$ is convex on $X$. For any $(u,u^*)\in (U_{\e}\times V)\cap \gph \partial  \varphi$ we get $(u,u^*)\in \gph \partial\widehat{\varphi}$, and therefore
\begin{equation*}
\widehat{\varphi}(x)\ge \widehat{\varphi}(u) + \la u^*,x-u\ra\;\text{ whenever }\;x\in U.
\end{equation*}
Since $\widehat{\varphi}(x)\le \varphi(x)$ and $\widehat{\varphi}(u)=\varphi(u)$, it follows that $\varphi(x) \ge \varphi(u) + \la u^*, x - u\ra$ for all $x \in U$, which justifies implication {\bf(i)}$\Longrightarrow${\bf(ii)}.\vspace*{0.05in}

To verify next the reverse implication {\bf(ii)}$\Longrightarrow${\bf(i)}, observe first that having ${\bf(ii)}$ allows us to shrink $U$ and $V$ if necessary so that $U\subset B_{r_1}(\ox)$ and $V\subset B_{r_2}(\ox^*)$ with $r_1 < \varepsilon/(\|\ox^*\|+r_2)$. Define the function $\widehat{\varphi}:X \rightarrow \overline{\R}$ by 
\begin{equation}\label{phimu}
\widehat{\varphi}(x): = \sup\big\{\varphi(u) + \la u^*, x-u \ra\;\big|\;(u,u^*)\in (U_{\e}\times V)\cap \gph \partial \varphi\big\}, \quad x \in X. 
\end{equation}
Since $\widehat{\varphi}$ is the supremum of a collection of affine functions, it is l.s.c.\ and convex on the whole space $X$. We clearly have the estimate
\begin{equation*}
\widehat{\varphi}(x)\ge \varphi(\ox)+\langle \ox^*,x-\ox\rangle >-\infty\;\mbox{ whenever }\;x\in X.
\end{equation*}
Moreover, it follows from \eqref{ii} and the construction of $\widehat{\varphi}$ that $\widehat{\varphi}\le \varphi$ on $U$. We split the remaining proof of {\bf(ii)}$\Longrightarrow${\bf(i)} into the {\em six steps/claims} as follows. 

\medskip\noindent
\textbf{Claim~1:} \textit{In any Banach space $X$, the inclusion $(x,x^*)\in (U_{\varepsilon}\times V)\cap \gph \partial \varphi$ is equivalent to having $(x,x^*)\in (U\times V)\cap \gph \partial \widehat{\varphi}$ together with $\widehat{\varphi}(x)=\varphi(x)$ for any l.s.c.\ function $\ph\colon X\to\oR$.}\\[1ex]
To justify this claim, consider $(x,x^*)\in (U_{\varepsilon}\times V)\cap \gph \partial \varphi$ and check first that $\widehat{\varphi}(x)=\varphi(x)$. Indeed, $\varphi(x)$ can be equivalently written as 
\begin{equation*}
\varphi(x) = \varphi(x) + \langle x^*, x - x \ra\;\text{ with }\;(x,x^*) \in (U_{\varepsilon}\times V)\cap \gph \partial \varphi,
\end{equation*}
which implies that $\widehat{\varphi}(x) \ge \varphi(x)$. Combining the latter with the inequality $\widehat{\varphi}\leq \varphi$ on $U$, we arrive at $\widehat{\varphi}(x)=\varphi(x)$. To check further that $(x,x^*)\in \gph \partial \widehat{\varphi}$, deduce from $(x,x^*)\in (U_{\varepsilon}\times V)\cap \gph \partial \varphi$ that for any $x^{\prime}\in X$ it follows that
\begin{equation*}
\langle x^*, x^{\prime} -x\rangle + \widehat{\varphi}(x) =  \langle x^*, x^{\prime} -x\rangle + \varphi(x) 
\le \widehat{\varphi}(x^\prime),
\end{equation*}
where the last inequality is a consequence of \eqref{phimu}. This tells us that $\langle x^*, x^\prime -x\rangle  \leq \widehat{\varphi}(x^\prime)-\widehat{\varphi}(x)$ for all $x^\prime \in X$, i.e.,  $x^* \in \partial \widehat{\varphi}(x)$.\vspace*{0.05in}

To verify the reverse implication in Claim~1, pick $(x,x^*)\in(U\times V) \cap \gph \partial  \widehat{\varphi}$ such that $\widehat{\varphi}(x)=\varphi(x)$. We aim to show that $(x,x^*)\in (U_{\varepsilon}\times V)\cap \gph \partial \varphi$. By $x^*\in \partial  \widehat{\varphi}(x)$, we have 
\begin{equation*}
\langle x^*, x^\prime  -x\rangle  + \widehat{\varphi}(x)\leq \widehat{\varphi}(x^\prime)\;\text{ for all }\;x^\prime  \in X.
\end{equation*}
It follows from  $\widehat{\varphi}\le \varphi$ on $U$ that
\begin{equation}\label{1}
\langle x^*, x^\prime - x\rangle  + \varphi(x) \le \varphi(x^\prime) \ \text{ whenever } \ x^\prime  \in U,
\end{equation}
and therefore we have the limiting condition
$$
\liminf_{x^\prime \to x} \frac{\varphi(x^\prime)-\varphi(x)-\langle  x^*,x^\prime -x\rangle}{\|x^\prime  - x\|} \ge 0,
$$
which means that $x^*\in \widehat{\partial}\varphi(x)$ and hence $x^*\in \partial \varphi(x)$. It remains to check that $x\in U_{\varepsilon}$, i.e., $\varphi(x) < \varphi(\ox) + \varepsilon$. Indeed, it follows from \eqref{1} and the choice of $U$ and $V$ above that
\begin{equation*}
\begin{array}{ll}
\varphi(x) \le \varphi(\ox)-\langle  x^*, \ox - x\rangle   \le \varphi(\ox) + \|x^*\| \cdot \|\ox - x\| \\ 
\le \varphi(\ox) + (\|x^*-\ox^*\|+\|\ox^*\|)\cdot \|\ox-x\| \le \varphi(\ox) + (r_2 + \|\ox^*\|) r_1 < \varphi(\ox) + \varepsilon,
\end{array}
\end{equation*}
which completes the proof of Claim~1.\vspace*{0.05in}

To proceed further, select a bounded, closed, and convex set $\O\subset U$ containing $\ox$ as an interior point and such that $\varphi$ is weakly sequentially l.s.c.\ at any point in $\O$. For $(x,x^*)\in \gph \partial \widehat{\varphi}$, define the auxiliary l.s.c.\ function $\psi_{x,x^*}$ by
\begin{equation}\label{psi}
\psi_{x,x^*}(y): = \varphi (y) - \widehat{\varphi}(x) - \la x^*, y - x \ra +\dfrac{1}{2}\|y-x\|^2 + \delta_\O (y),\quad y \in X,
\end{equation}
via the {\em indicator function} $\dd_\O(x)$ of $\O$ equal to $0$ if $x\in\O$ and $\infty$ otherwise.

\medskip\noindent
\textbf{Claim~2:} \textit{Let $X$ be a Banach space, and let $\psi_{x,x^*}$ be taken from \eqref{psi}. Then for any $(x,x^*)\in \gph \partial \widehat{\varphi}$ and any $y\in X$, we have $\psi_{x,x^*}(y)\ge 0$. Furthermore, it follows that}
\begin{equation*}
{\rm argmin}\,\psi_{\ox,\ox^*}=\{\ox\}\;\mbox{ and }\;\min\psi_{\ox,\ox^*}=0.
\end{equation*}
Indeed, picking $(x,x^*)\in \gph \partial \widehat{\varphi}$ ensures that $\la x^*, y-x\ra +\widehat{\varphi}(x) \le \widehat{\varphi}(y) \le \varphi(y)$ whenever $y\in U$. Since $\O\subset U$, it follows that $\psi_{x,x^*}(y) \ge 0$ for all $y\in X$, and that $\psi$ vanishes if and only if
\begin{eqnarray*}
\begin{cases}
y=x, \\
y\in\O, \\
\varphi(y)=\widehat{\varphi}(x).
\end{cases}
\end{eqnarray*}
Applying this to $(\ox,\ox^*)$ gives us $\argmin\psi_{\ox,\ox^*}=\{\ox\}$ and $\min \psi_{\ox,\ox^*}=0$ as stated in Claim~2.

\medskip\noindent
\textbf{Claim~3:} {\em Let $X$ be a reflexive Banach space, and let $\ph\colon X\to\oR$ be weakly sequentially l.s.c.\ around $\ox$. Then each function $\psi_{x,x^*}$ achieves its global minimum on $X$.}\\[1ex]
Since $\O$ is closed and convex, this set is weakly closed by the classical Mazur theorem, and hence it is weakly sequentially closed. On the other hand, we know that the function $\la x^*,\cdot -x\ra$ is weakly continuous, while $\|\cdot-x\|^2$ is weakly sequentially l.s.c.\ around $\ox$. Also, the indicator function $\delta_\O(\cdot)$ is weakly sequentially l.s.c.\ at any point in $\O$. Combining this with the imposed weakly sequentially l.s.c.\ property of $\varphi$ tells us that the function $\psi_{x,x^*}$ is weakly sequentially l.s.c.\ on the bounded and weakly sequentially closed subset $\O$ of the reflexive Banach space $X$. Thus the generalized Weierstrass theorem in this setting (see, for instance, \cite[Theorem~7.3.4]{FA-Kur}) ensures that each $\psi_{x,x^*}$ attains its minimum on $\O$, which is actually a global minimum on $X$ due to the construction of $\psi_{x,x^*}$ in \eqref{psi}. We are done with the proof of Claim~3.

\medskip\noindent
\textbf{Claim~4:} {\em Let $X$ be a Banach space, and let $\ph\colon X\to\oR$ be l.s.c.\ on $X$. Then the function $\psi_{x,x^*}$ depends {\sc epicontinuously} on $(x,x^*)$ with respect to $\gph \partial \widehat{\varphi}$, in the sense that the convergence  $(x_k,x_k^*)\xrightarrow{\gph \partial \widehat{\varphi}}(x,x^*)$ implies that $\psi_{x_k,x_k^*} \xrightarrow{e}\psi_{x,x^*}$ as $k\to\infty$.} 
\\[1ex]
To justify this claim, take a sequence $(x_k,x_k^*)\xrightarrow{\gph \partial \widehat{\varphi}}(x,x^*)$ as $k\to\infty$ and verify the epi-convergence $\psi_{x_k,x_k^*} \xrightarrow{e} \psi_{x,x^*}$  by using its characterization given in Lemma~\ref{lem:epi-charac}. To this end, pick any $y\in X$ and any sequence $y_k\rightarrow y$. Since $\widehat{\varphi}$
is l.s.c.\ and convex on $X$, it is subdifferentially continuous at $x\in \dom \widehat{\varphi}$; see, e.g., \cite[Example~13.30]{Rockafellar98} which proof works in arbitrary Banach spaces. Combining this with $(x_k,x_k^*)\xrightarrow{\gph \partial\widehat\varphi} (x,x^*)$ gives us the convergence $\widehat{\varphi}(x_k)\rightarrow \widehat{\varphi}(x)$ as $k\to\infty$. Therefore, we readily get the relationships
\begin{eqnarray*}
\begin{array}{ll}
\disp\liminf_{k\to\infty}\psi_{x_k,x_k^*}(y_k)\ge\disp\liminf_{k\to\infty}
\varphi(y_k)+\disp\liminf_{k\to\infty}\delta_\O(y_k)-\disp\lim_{k\to\infty}\widehat{\varphi}(x_k)\\
+\disp\lim_{k\to\infty}\Big[\dfrac{1}{2}\|y_k-x_k\|^2 -\la x^*_k, y_k - x_k\ra \Big]\ge\varphi (y) + \delta_\O(y) - \widehat{\varphi}(x) + \dfrac{1}{2}\|y-x\|^2 - \la x^*,y-x\ra\\
=\psi_{x,x^*}(y).
\end{array}
\end{eqnarray*} 
On the other hand, choosing the constant sequence $z_k:= y$ tells us that
\begin{eqnarray*}
\begin{array}{ll}
\disp\limsup_{k\to\infty}\psi_{x_k,x_k^*}(y)=\disp\limsup_{k\to\infty} \varphi(y)+\disp\limsup_{k\to\infty} \delta_\O(y)-\disp\lim_{k\to\infty} \widehat{\varphi}(x_k)\\
+\disp\lim_{k\to\infty}\Big[\dfrac{1}{2}\|y-x_k\|^2 -\la x^*_k, y -x_k\ra \Big]=\varphi(y) + \delta_\O(y) - \widehat{\varphi}(x) + \dfrac{1}{2}\|y-x\|^2 - \la x^*,y-x\ra \\
=\psi_{x,x^*}(y),
\end{array}
\end{eqnarray*}
which ensures that $\psi_{x_k,x_k^*}\xrightarrow{e}\psi_{x,x^*}$ as $k\to\infty$ and thus verifies the claim.

\medskip\noindent
\textbf{Claim~5:} {\em Let $X$ be a reflexive Banach  space. Suppose that the sequence $(x_k,x_k^*)\in (U\times V)\cap \gph \partial \widehat{\varphi}$ converges to some 
$(\ox,\ox^*)$ as $k\to\infty$ and that the function $\ph$ is weakly sequentially l.s.c.\ around $\ox$. Then for each $k\in\N$ sufficiently large, there exists a pair $(\tilde{x}_k,\tilde{x}^*_k)\in (U_{\varepsilon}\times V)\cap \gph \partial \varphi$ such that $\widehat{\varphi}(\tilde{x}_k)=\varphi (\tilde{x}_k)$ and $x_k^* - \tilde{x}^*_k \in J(\tilde{x}_k - x_k)$.} \\[1ex]
Indeed, it follows from Claim~4 that $\psi_{x_k,x_k^*}\xrightarrow{e}\psi_{\ox,\ox^*}$ as $k\to\infty$. Taking into account that $\argmin\psi_{x_k,x_k^*}\ne\emp$ by Claim~3 and employing Lemma~\ref{Lem:fixed1} allow us to construct a sequence $\{\tilde{x}_k\}$ such that $\tilde{x}_k \rightarrow \ox$ as $k\to\infty$ and that
\begin{eqnarray*}
\tilde{x}_k \in \argmin\psi_{x_k,x^*_k} \cap \mathrm{int}\,\O\;\text{ for all large }\;k\in\N. 
\end{eqnarray*}
Lemma~\ref{Lem:2.9-Attouch} gives us the relationships 
\begin{equation*}
0=\min\psi_{\ox,\ox^*} \ge \limsup_{k\to\infty}(\inf\psi_{x_k,x_k^*}) = \limsup_{k\to\infty} \big(\psi_{x_k,x_k^*}(\tilde{x}_k)\big),
\end{equation*}
which ensure that $\psi_{x_k,x_k^*}(\tilde{x}_k)\searrow 0$ due to $\psi_{x_k,x_k^*}(\tilde{x}_k) \ge 0$ for all $k\in\N$. Thus for sufficiently large $k\in \N$, we have the conditions
\begin{equation*}
\varphi(\tilde{x}_k) - \widehat{\varphi}(x_k)=\psi_{x_k,x_k^*}(\tilde{x}_k) + \la x^*_k, \tilde{x}_k - x_k\ra - \dfrac{1}{2}\|\tilde{x}_k - x_k\|^2  < \dfrac{\e}{2}.
\end{equation*}
The latter clearly implies the estimate
\begin{eqnarray*}
\varphi(\tilde{x}_k) &<& \widehat{\varphi}(x_k) + \dfrac{\e}{2} \\ 
&\le& \widehat{\varphi}(\ox) + \e \quad \text{(since $\widehat{\varphi}$ is subdifferentially continuous at $\ox$)} \\
&=& \varphi(\ox) + \e,
\end{eqnarray*}
which tells us that $\tilde{x}_k\in U_{\e}$ for sufficiently large $k$. Remembering that $\tilde{x}_k$ is a minimizer of $\psi_{x_k,x^*_k}$ on $X$, we get by the subdifferential Fermat rule (see, e.g., \cite[Proposition~1.114]{Mordukhovich06}) that $0\in \partial\psi_{x_k,x^*_k}(\tilde x_k)$. Since $\ph$ is l.s.c.\ while all other functions in \eqref{psi} are locally Lipschitzian around $\ox$ due to the construction of $\psi_{x_k,x^*_k}$, and since any reflexive Banach space is Asplund, we apply to $\partial\psi_{x_k,x^*_k}(\tilde x_k)$ the semi-Lipschitzian subdifferential sum rule from \cite[Theorem~2.33]{Mordukhovich06} to obtain
\begin{eqnarray*}
0 \in \partial\psi_{x_k,x^*_k}(\tilde{x}_k) 
\subset \partial \varphi(\tilde{x}_k) - x^*_k + J(\tilde{x}_k-x_k)\;\mbox{ for all large }\;k\in\N.
\end{eqnarray*}
Observe that $\partial(\|\cdot-x_k\|^2)(\tilde{x}_k) = J(\tilde{x}_k-x_k)$ via the duality mapping \eqref{duality} and that $\partial \delta_\O(\tilde{x}_k)=\{0\}$ due to the choice of $\ox$ and the convergence $\tilde{x}_k\to\ox$. Thus there exist $\tilde{x}^*_k \in \partial \varphi(\tilde{x}_k)$ and $j(\tilde{x}_k-x_k) \in J(\tilde{x}_k-x_k)$ such that $\tilde{x}^*_k = x^*_k - j(\tilde{x}_k-x_k)$. Since $(x_k,x_k^*)\rightarrow (\ox,\ox^*)$ and $\tilde{x}_k \rightarrow \ox$, we also have $\tilde{x}_k - x_k \rightarrow 0$ as $k\to\infty$. By the fact that $\|j(\tilde{x}_k-x_k)\| = \|\tilde{x}_k - x_k\|$, this gives us $j(\tilde{x}_k-x_k) \rightarrow 0$, and so $\tilde{x}_k^* \rightarrow \ox^*$ as $k\to\infty$. In summary, for all $k$ sufficiently large, there are $(\tilde{x}_k,\tilde{x}_k^*)\in \gph \partial \varphi$ such that $\tilde{x}_k \in U_{\e}$ and $(\tilde{x}_k,\tilde{x}_k^*)\rightarrow (\ox,\ox^*)$ as $k\to\infty$. Therefore, $(\tilde{x}_k,\tilde{x}_k^*)\in (U_{\e}\times V)\cap \gph \partial \varphi$. Consequently, $(\tilde{x}_k,\tilde{x}_k^*)\in \gph \partial \widehat{\varphi}$ and $\widehat{\varphi}(\tilde{x}_k)=\varphi(\tilde{x}_k)$ for large $k$ due to Claim~1. This verifies the formulated assertions.

\medskip\noindent
\textbf{Claim~6:} {\em In the setting of Claim~{\rm 5}, we have $\tilde{x}_k=x_k$ for all large $k\in \N$.}\\[1ex]
Indeed, the subgradient mapping $\partial \widehat{\varphi}$ is monotone due to the properness and convexity of $\widehat{\varphi}$. Since both pairs $(x_k,x^*_k)$ and $(\tilde{x}_k,\tilde{x}^*_k)$ belong to $\gph \partial \widehat{\varphi}$ due to Claim~5, we get that 
\begin{equation*}
0 \le \la x^*_k - \tilde{x}^*_k , x_k - \tilde{x}_k \ra = -\|x_k-\tilde{x}_k\|^2,
\end{equation*}
and thus $\tilde{x}_k=x_k$, which verifies this claim.\vspace*{0.05in}

Unifying the above claims with shrinking the neighborhoods $U$ and $V$ if necessary tell us that for all $(x,x^*)\in (U\times V)\cap \gph \partial \widehat{\varphi}$ we have $\widehat{\varphi}(x)=\varphi(x)$. This completes the proof of implication {\bf(ii)}$\Longrightarrow${\bf(i)} and hence of the whole theorem.
\end{proof}

\section{Characterizations of Variational Convexity via Local Monotonicity and Moreau Envelopes}\label{moreau}

This section provides new characterizations of variational convexity of extended-real-valued l.s.c.\ functions on Banach spaces. The first type of characterizations involves {\em local monotonicity} and {\em local maximal monotonicity} of {\em limiting subgradient mappings}, while the second type characterizes variational convexity of functions via the standard {\em local convexity} of their {\em Moreau envelopes}.\vspace*{0.05in}

Given an l.s.c.\ function $\varphi:X \to  \overline{\R}$ on a Banach space $X$ together with a parameter $\lambda >0$, the \textit{Moreau envelope} $e_{\lambda} \varphi:X\to \overline{\R}$ and the \textit{proximal mapping} $P_{\lambda} \varphi: X \rightrightarrows X$ associated with $\ph$ and $\lm$ are defined, respectively, by
\begin{equation}\label{el}
e_\lambda \varphi(x) := \inf_{w\in X} \left\{ \varphi(w) + \dfrac{1}{2\lambda}\|w-x\|^2\right\},\quad x\in X,
\end{equation}   
\begin{equation}\label{pl}
 P_\lambda \varphi(x) := \text{\rm argmin}_{w\in X} \left\{ \varphi(w) + \dfrac{1}{2\lambda}\|w-x\|^2\right\},\quad x\in X.
\end{equation}

Prior to establishing the main characterizations of this section, we present several lemmas, which are of their own interest while being instrumental for the proof of the major results below. The first lemma provides a description of the {\em prox-boundedness} \eqref{prox-bounded} of extended-real-valued functions via their Moreau envelopes \eqref{el}.

\begin{Lemma}\label{prop:charac-prox-bound}
Let $\varphi:X\rightarrow \overline{\R}$ be an l.s.c.\ function defined on a Banach space $X$. Then the following assertions are equivalent:
\begin{enumerate}
\item[{\bf(i)}] $\varphi$ is prox-bounded.
\item[{\bf(ii)}] There exist $\lambda >0$ and $x\in X$ such that $e_{\lambda} \varphi (x) >-\infty$.
\item[{\bf(iii)}] There exists $\lambda_0 >0$ such that $e_{\lambda}\varphi (x)>-\infty$ for any $0<\lambda<\lambda_0$ and for any $x\in X$.
\end{enumerate}
\end{Lemma}
\begin{proof}
We begin with verifying implication {\bf(i)}$\Longrightarrow${\bf(ii)}. The prox-boundedness of $\ph$ gives us $\alpha$, $\beta\in \R$ and $x\in X$ such that 
\begin{equation*}
\varphi (y) \ge \alpha \|y-x\|^2 + \beta\;\text{ for all }\;y\in X.
\end{equation*}
Then for any $\lambda>0$ satisfying $\alpha+\dfrac{1}{2\lambda}\ge 0$, we get 
\begin{equation*}
\varphi(y) + \dfrac{1}{2\lambda}\|y-x\|^2 \ge \Big(\alpha+\dfrac{1}{2\lambda}\Big) \|y-x\|^2 +\beta\;\text{ whenever }\;y\in X.
\end{equation*}
Taking the infimum of both sides above with respect to $y\in X$ brings us to  $e_{\lambda} \varphi(x)\ge \beta >-\infty$, which therefore verifies assertion {\bf(ii)}.

Suppose next that {\bf(ii)} holds. Then denote 
\begin{equation*}
\lambda_0:= \sup\big\{\lambda >0\;\big|\;\exists x\in X\;\text{ and }\; e_{\lambda}\varphi(x)>-\infty\big\}
\end{equation*}
and deduce from {\bf(ii)} that $\lambda_0 >0$. Consider any $0<\lambda<\lambda_0$ and any $x\in X$. Since $\lambda <\lambda_0$, there exists $\bar{\lambda} > \lambda$ with $\ox\in X$ satisfying $e_{\bar{\lambda}}\varphi(\ox)>-\infty$. For $y\in X$, we clearly have that
\begin{eqnarray*}
 \varphi(y)+\dfrac{1}{2\lambda}\|y-x\|^2 &=& \varphi(y)+\dfrac{1}{2\bar{\lambda}}\|y-\ox\|^2 + \dfrac{1}{2\lambda}\|y-x\|^2 -\dfrac{1}{2\bar{\lambda}}\|y-\ox\|^2 \\
&\ge& e_{\bar{\lambda}}\varphi(\ox) + \dfrac{1}{2\lambda}\|y-x\|^2 -\dfrac{1}{2\bar{\lambda}}(\|y-x\|+\|x-\ox\|)^2 \\
&=& e_{\bar{\lambda}}\varphi(\ox) + \left(\dfrac{1}{2\lambda}-\dfrac{1}{2\bar{\lambda}}\right)\cdot\|y-x\|^2 -\dfrac{1}{\bar{\lambda}}\|y-x\|\cdot\|x-\ox\|-\dfrac{1}{2\bar{\lambda}}\|x-\ox\|^2.
\end{eqnarray*}
Since $e_{\bar{\lambda}}\varphi(\ox)>-\infty$ and $\bar{\lambda}>\lambda$, the expression $\varphi(y)+\dfrac{1}{2\lambda}\|y-x\|^2$ is bounded below when $y$ varies in $X$. This gives us the conditions
\begin{equation*}
e_{\lambda}\varphi (x) = \inf_{y\in X} \left\{ \varphi (y)+\dfrac{1}{2\lambda}\|y-x\|^2 \right\} >-\infty,
\end{equation*}
which verify {\bf(iii)}. The remaining implication {\bf(iii)}$\Longrightarrow${\bf(i)} is trivial, and thus we are done.
\end{proof}

Given $x^*\in X^*$ and $\lambda >0$, define the {\em $x^*$-Moreau envelope} for $\varphi\colon X\to\oR$ by
\begin{equation}\label{defi:Motilted}
e^{x^*}_{\lambda} \varphi (x):= \inf_{w\in X} \left\{ \varphi(w) - \la x^*,w\ra +\dfrac{1}{2\lambda}\|w-x\|^2\right\},\quad x\in X,
\end{equation}
and the {\em $x^*$-proximal mapping} associated with $\ph$ by 
\begin{equation}\label{defi:proxtilted}
P_{\lambda}^{x^*} \varphi(x) := \text{\rm argmin}_{w\in X} \left\{ \varphi(w) - \la x^*,w\ra +\dfrac{1}{2\lambda}\|w-x\|^2\right\},\quad x\in X.
\end{equation}
It is easy to see that the $x^*$-Moreau envelope for $\varphi$ at $x\in X$ is the standard Moreau envelope for the \textit{tilted function} $\varphi (\cdot)-\langle x^*,\cdot\rangle$ at the same point. Therefore, we have the following relationships between $x^*$-Moreau envelopes (resp.\ $x^*$-proximal mappings) and standard Moreau envelopes (resp.\ proximal mappings) in the setting of Hilbert spaces:
\begin{equation}\label{rela:Moreau}
e_{\lambda}^{\ox^*}\varphi (\ox) = e_{\lambda}\varphi (\ox+\lambda \ox^*) -\langle \ox^*,\ox\rangle -\frac{\lambda}{2}\|\ox^*\|^2\quad \text{ and }\quad
P_{\lambda}^{\ox^*}\varphi (\ox) = P_{\lambda}\varphi (\ox + \lambda \ox^*).
\end{equation}
This result can be found in \cite[Lemma~2.2]{Poliquin04} which proof works for functions on Hilbert spaces.\vspace*{0.05in} 

The next lemma provides characterizations of {\em fixed points} of the $x^*$-proximal mappings defined in \eqref{defi:proxtilted}.  Recall from \cite[Definition~8.45]{Rockafellar98} that $\ox^*\in X^*$ is a \textit{proximal subgradient} of $\varphi$ at $\ox\in \dom \varphi$ if there exist numbers $r>0$ and $\varepsilon>0$ such that for all $x \in B_{\e}(\ox)$ we have 
\begin{equation}\label{fixed1}
\varphi(x)\ge \varphi(\ox) + \la \ox^*, x - \ox\ra - \dfrac{r}{2}\|x - \ox\|^2.
\end{equation}

\begin{Lemma}\label{prop:Patxbar} Consider an l.s.c.\ function $\varphi:X\to \overline{\R}$ with $\ox\in \dom \varphi$, and $\ox^*\in X^*$. Then the following statements are equivalent:
\begin{enumerate}
\item[{\bf(i)}] $\varphi$ is prox-bounded, and $\ox^*$ is a proximal subgradient of $\varphi$ at $\ox$.
\item [{\bf(ii)}] $P^{\ox^*}_{\lambda} \varphi (\ox)=\{\ox\}$ for some $\lambda >0$.
\item[{\bf(iii)}] $P^{\ox^*}_{\lambda} \varphi (\ox)=\{\ox\}$ for all $\lm>0$ sufficiently small.
\end{enumerate}
\end{Lemma}
\begin{proof}
Assume that {\bf(i)} holds. To verify {\bf(ii)}, take $\ox^*$ as a proximal subgradient of $\varphi$ at $\ox$ and then find $r,\varepsilon>0$ such that \eqref{fixed1} is satisfied. We intend to show that \eqref{fixed1} holds for all $x \in X$ with a certain modification of $r>0$. Since $\varphi$ is prox-bounded, Lemma~\ref{prop:charac-prox-bound} implies that there exists $\lambda>0$ with $e_{\lambda}\varphi(\ox)>-\infty$. On the other hand, choosing $\bar{\lambda}>0$ to be sufficiently small guarantees that the estimates
\begin{equation}\label{K3}
\varphi(\ox)+\langle  \ox^*,x -\ox\rangle - \frac{1}{2\bar{\lambda}}\|x -\ox\|^2 \le  e_\lambda \varphi(\ox) - \frac{1}{2\lambda}\|x -\ox\|^2 \le \varphi (x)
\end{equation}
hold for all $x$ satisfying $\|x -\ox\|>\varepsilon$. Denote $\overline{r}:=2\max \{r,1/\bar{\lambda}\}$ and deduce from \eqref{fixed1} and \eqref{K3} the fulfillment of the strict inequality
\begin{equation*}
\varphi(x)-\la \ox^*,x\ra + \dfrac{\overline{r}}{2}\|x-\ox\|^2 > \varphi(\ox) - \la \ox^*,\ox\ra \ \text{ for all } \ x  \neq \ox .
\end{equation*}
The latter is equivalent to saying that $P_{\bar{\lambda}}^{\ox^*}\varphi(\ox)=\{\ox\}$, where $\bar{\lambda}:=1/\overline{r}$, and hence {\bf(ii)} is verified.

Implication {\bf(ii)}$\Longrightarrow${\bf(iii)} is trivial. Finally, suppose that {\bf(iii)} holds and deduce from $P^{\ox^*}_{\lambda}\varphi (\ox) = \{\ox\}$ that $e^{\ox^*}_{\lambda}\varphi (\ox) >-\infty$, which yields the prox-boundedness of $\varphi$ by Lemma~\ref{prop:charac-prox-bound}. The assertion that $\ox^*$ is a proximal subgradient of $\varphi$ at $\ox$ is implied by the fact that $\ox$ is a global minimizer of $\varphi (\cdot) - \la \ox^* , \cdot\ra + \dfrac{1}{2\lambda}\|\cdot - \ox\|^2$. Thus we arrive at {\bf(i)} and complete the proof of the lemma.
\end{proof}

The following lemma is a combination of the results taken from \cite[Theorems~5.3 and 5.5]{Thibault04}.

\begin{Lemma}\label{theo:5.3-Thibault} Let $X$ be a $2$-uniformly convex space which norm is G\^ateaux differentiable on $X\setminus\{0\}$, and let $\varphi: X \rightarrow \overline{\R}$ be an l.s.c., prox-bounded, and prox-regular function at $\ox$ for $\ox^*\in \partial \varphi (\ox)$. Then there exist $\lambda_0>0$ and $\varepsilon>0$ such that for any positive number $\lambda\le\lambda_0$, there is some neighborhood $U_\lambda$ of $\ox$ on which $e^{\ox^*}_{\lambda}\varphi$ is $\mathcal{C}^1$-smooth with
\begin{equation*}
\nabla e_\lambda^{\ox^*} \varphi=\lambda^{-1} J \circ (I-P_\lambda^{\ox^*} \varphi),
\end{equation*}
while the proximal mapping $P_\lambda^{\ox^*} \varphi$ is single-valued, continuous, and admits the representation
$$
P_\lambda^{\ox^*} \varphi(u)=\big(I+\lambda J^{-1} \circ\big(T_{\varepsilon}^\varphi-\ox^*\big)\big)^{-1}(u)\;\text { for all }\;u \in U_\lambda,
$$
where $T_{\varepsilon}^\varphi$ is the $\ph$-attentive $\ve$-localization of $\partial\ph$ around $(\ox,\ox^*)$ taken from \eqref{localization}, and where the mapping $\big(T_{\varepsilon}^\varphi-\ox^*\big)(x):=T_{\varepsilon}^\varphi(x)-\ox^*$ is also single-valued on $U_\lambda$.
\end{Lemma} 

The next result reveals important properties of the duality mapping \eqref{duality}, which are instrumental in the proof of the main result below.

\begin{Lemma}\label{lem:J1} Let $X$ be a uniformly convex Banach space which norm is G\^ateaux differentiable at nonzero points of $X$. Then the corresponding duality mapping $J$ is single-valued and norm-to-norm continuous on the whole space $X$.
\end{Lemma}
\begin{proof}
Fix an arbitrary point $x\in X$ and take a sequence $x_k \to x$ as $k\to\infty$. By the Milman-Pettis theorem, the space $X$ is reflexive. Due to the reflexivity of $X$ and the assumed smoothness of the norm $\|\cdot\|$, it follows from \cite[Corollary~1.4]{Ciora} that the dual space $X^*$ is {\em strictly convex}, i.e., its closed unit ball is a strictly convex set. Moreover, \cite[Corollary~1.5]{Ciora} tells us that $J$ is single-valued and norm-to-weak$^*$ continuous. This implies that $J(x_k) \xrightarrow{w^*} J(x)$ as $k\to\infty$. Since $X$ is reflexive, we actually have $J(x_k) \xrightarrow{w} J(x)$. Furthermore, the convergence $\|x_k\|\to \|x\|$ and definition \eqref{duality} of the duality mapping $J$ ensure that $\|J(x_k)\|\to\|J(x)\|$. Finally, we get by \cite[Proposition~3.32]{FA} that $J(x_k)\to J(x)$ as $k\to\infty$, which thus completes the proof.   
\end{proof}

Before deriving the major relationships for variational convexity of functions defined on Banach spaces, we formulate the modified notions of local monotonicity and local maximal monotonicity of subgradient mappings by following the finite-dimensional pattern of \cite{rtr251}.

\begin{Definition}\label{ph-local mon}\rm Given an l.s.c.\ function $\varphi: X \rightarrow \overline{\R}$ defined on a Banach space $X$, we say that the subgradient mapping $\partial \varphi: X \rightrightarrows X^*$ is {\em $\varphi$-locally  monotone} $(${\em $\varphi$-locally maximal monotone$)$} around $(\bar{x}, \bar{x}^*)\in \gph \partial\varphi$ if there exist a neighborhood $U \times V$ of $(\bar{x}, \bar{x}^*)$ and a number $\ve>0$ such that the mapping
$\partial \varphi$ is locally monotone $($locally maximal monotone$)$ with respect to $U_{\varepsilon} \times V$, where $U_{\varepsilon}$ is taken from Definition~{\rm\ref{def:vc}}.
\end{Definition}

Here are the aforementioned major results on variational convexity of extended-real-valued l.s.c.\ functions defined on Banach spaces. The equivalence between assertions {\bf(i)}, {\bf(ii)}, and {\bf(iii)} of the next theorem was established in \cite{rtr251} in finite-dimensional spaces. In \cite{kmp22convex}, the variational convexity of an l.s.c.\ prox-bounded function $\ph$ at $\ox$ for $\ov\in\partial\ph(\ox)$ was characterized in finite dimensions by the local convexity of its Moreau envelope $e_\lm\ph$ from \eqref{el} around $\ox+\lm\ov$ for small $\lm>0$. Our theorem below offers more variety in infinite dimensions by involving geometric properties of Banach spaces. Note, in particular, that---outside of Hilbert spaces---we now replace in assertion {\bf(iv)} below the local convexity of the Moreau envelope $e_\lm\ph$ around $\ox+\lm\ov$ by the local convexity of the {\em tilted} $x^*$-Moreau envelope from \eqref{defi:Motilted} around $\ox$.

\begin{Theorem}\label{1stequi} Let $X$ be a Banach space, and let $\varphi:X \to \overline{\R}$ be an l.s.c. function with $\ox^*\in \partial \varphi(\ox)$. Consider the following assertions:
\begin{itemize}
\item[\bf(i)] $\varphi$ is variationally convex at $\ox$ for $\ox^*$. 
\item[\bf(ii)] $\partial \varphi$ is $\varphi$-locally maximal monotone around $(\ox,\ox^*)$. 
\item[\bf(iii)] $\partial \varphi$ is $\varphi$-locally monotone around $(\ox,\ox^*)$.
\item[\bf(iv)] The $\ox^*$-Moreau envelope $e^{\ox^*}_{\lambda}\varphi$ is locally convex around $\ox$ for all small $\lambda>0$. 
\end{itemize}
Then we have implications  {\bf(i)}$\Longrightarrow${\bf(ii)}$\Longrightarrow${\bf (iii)} in general Banach spaces. Implication {\bf(iii)}$\Longrightarrow${\bf(iv)} holds when $X$ is a $2$-uniformly convex space which norm is smooth, while 
$\varphi$ is prox-bounded and prox-regular at $\ox$ for $\ox^*$. Finally, implication {\bf(iv)}$\Longrightarrow${\bf (i)} is satisfied if in addition the function $\ph$ is weakly sequentially l.s.c.\ around $\ox$.
\end{Theorem}
\begin{proof} 
First we verify implication {\bf(i)$\Longrightarrow$(ii)} under the general assumptions on $\ph$ and $X$. It follows from the variational convexity of $\ph$ at $\ox$ for $\ox^*$ that there exist $\varepsilon>0$, a neighborhood $U\times V$ of $(\ox,\ox^*)$, and an l.s.c.\ convex function $\widehat{\varphi}:X\rightarrow \overline{\R}$ such that 
\begin{equation}\label{vc-1}
\gph \partial\widehat{\varphi} \cap (U\times V) = \gph \partial \varphi \cap (U_{\varepsilon}\times V).
\end{equation}
The classical result of \cite[Theorem~A]{Rockafellar70-Pac} applied to $\widehat{\varphi}$ on a Banach space $X$ tells us that the subgradient mapping $\partial\widehat{\varphi}$ is (globally) maximal monotone on $X$. Combining the latter with \eqref{vc-1} implies that the mapping $\partial \varphi$ is locally maximal monotone with respect to $U_{\varepsilon}\times V$, which thus justifies {\bf(ii)} by Definition~\ref{ph-local mon}. The next implication {\bf (ii)$\Longrightarrow$(iii)} is trivial.

Before verifying implications {\bf(iii)}$\Longrightarrow${\bf(iv)} and {\bf(iv)}$\Longrightarrow${\bf(i)}, we deduce from Lemma~\ref{theo:5.3-Thibault} that if $X$ is smooth and $2$-uniformly convex and $\varphi$ is prox-bounded and prox-regular, then there exist $\lambda_0>0$, $\gamma >0$, and a $\varphi$-attentive $\gamma$-localization $T^{\varphi}_{\gamma}: X\rightrightarrows X^*$ given by
\begin{equation}\label{att} 
T_{\gamma}^{\varphi}(x):=\begin{cases}\big\{x^*\in\partial\varphi(x)\big|\;\|x^*-\ox^*\|<\gamma\big\} & \text{if}\quad\|x-\ox\|<\gamma\;\text{ and }\;\varphi(x)<\varphi(\ox)+\gamma,\\
\emptyset  &\text{otherwise}
\end{cases}
\end{equation} 
such that for any $\lambda\in(0,\lambda_0)$ there is a neighborhood $U_\lambda$ of $\ox$ on which $e^{\ox^*}_\lambda\varphi$ is of class $\mathcal{C}^{1}$, that the $\ox^*$-proximal mapping $P_{\lambda}^{\ox^*}\varphi$ is single-valued and continuous on $U_\lambda$ with 
\begin{equation}\label{tilt1}
P_{\lambda}^{\ox^*}\varphi (u) = \big( I + \lambda J^{-1}\circ (T_{\gamma}^{\varphi}-\ox^*)\big)^{-1}(u) \text{ for all } u\in U_{\lambda},
\end{equation}
and that we have the gradient representation
\begin{equation}\label{tilt2}
\nabla e^{\ox^*}_{\lambda}\varphi = \lambda^{-1} J \circ (I-P_{\lambda}^{\ox^*}\varphi).
\end{equation}
To justify now implication {\bf(iii)}$\Longrightarrow${\bf(iv)}, suppose that $\partial\varphi$ is $\varphi$-locally monotone around $(\ox,\ox^*)$ and then find $\varepsilon>0$ and $r>0$ such that the subgradient mapping $\partial\varphi$ is monotone with respect to the set $W^{\e}:=B_{r}(\ox)^{\e}\times B_r(\ox^*)$, where
\begin{equation*}
B_r(\ox)^{\varepsilon}:=\big\{x\in B_r(\ox)\;\big|\;\varphi(x)<\varphi (\ox) + \varepsilon\big\}.
\end{equation*}
Since $\varphi$ is prox-bounded and $\ox^*$ is a proximal subgradient of $\varphi$ at $\ox$ by the assumed prox-regularity, it follows from Lemma~\ref{prop:Patxbar} that
\begin{equation}\label{I-P}
P^{\ox^*}_{\lambda}\varphi(\ox)=\{\ox\}\;\text{ and }\; e_{\lambda}^{\ox^*}\varphi (\ox)=\varphi (\ox)-\langle \ox^*,\ox\rangle.
\end{equation} 
Using the uniform convexity of $X$ together with the smoothness of its norm $\|\cdot\|$, we deduce from Lemma~\ref{lem:J1} that $J$ is norm-to-norm continuous. Considering further the function 
\begin{equation*}
\psi_\lambda:=e^{\ox^*}_\lambda\varphi(\cdot) +\big\langle \ox^*,P^{\ox^*}_{\lambda}\varphi (\cdot)\big\rangle\disp-\frac{1}{2\lambda}\|P_{\lambda}^{\ox^*}\varphi(\cdot)-\cdot \|^2,\quad x\in U_{\lambda}
\end{equation*}
and invoking \eqref{tilt2} and \eqref{I-P} give us the equalities
\begin{equation*}
\psi_{\lambda} (\ox) = e^{\ox^*}_\lambda\varphi(\ox) + \langle \ox^*,P^{\ox^*}_{\lambda}\varphi (\ox)\rangle   -\dfrac{1}{2\lambda}\|P_{\lambda}^{\ox^*}\varphi(\ox)-\ox \|^2 = \varphi (\ox),
\end{equation*}
\begin{equation*}
\nabla e_{\lambda}^{\ox^*} \varphi (\ox) = \lambda^{-1} J \circ\big(\ox - P_{\lambda}^{\ox^*}\varphi (\ox)\big) = 0.
\end{equation*}
It follows from \eqref{I-P} and the continuity of the mappings $\psi_{\lambda}$, $\nabla  e_{\lambda}^{\ox^*}\varphi$, and $P^{\ox^*}_{\lambda}\varphi$ around $\ox$ that
\begin{equation}\label{conti1}
P^{\ox^*}_{\lambda}\varphi (x) \in B_{r}(\ox),\quad    \psi_{\lambda}(x) < \varphi (\ox) + \varepsilon, \ \text{ and } \ \left\|\nabla e_{\lambda}^{\ox^*}\varphi (x)\right\| < r\;\text{ for all }\;x\in U 
\end{equation}
when a neighborhood $U\subset U_{\lambda}$ of $\ox$ is sufficiently small. 
The definitions in \eqref{defi:Motilted}, \eqref{defi:proxtilted} and the above construction of $\psi_\lambda$ ensure that $\psi_\lambda(x)=\varphi (P^{\ox^*}_{\lambda}\varphi (x))$ for all $x\in U$. Pick $x_i\in U$ and denote $y_i:=P_{\lambda}^{\ox^*}\varphi (x_i)$ for $i=1,2$. Employing the representations in \eqref{tilt1} and \eqref{tilt2} and remembering that the duality mapping $J$ is bijective in our setting bring us to
\begin{equation}\label{tilt22}
\lambda^{-1} J(x_i-y_i) +\ox^* \in T^{\varphi}_{\gamma}(y_i)\;\mbox{ and }\;\nabla e^{\ox^*}_{\lambda}\varphi (x_i)= \lambda^{-1} J(x_i-y_i),\quad i=1,2.
\end{equation}
It follows from \eqref{conti1} that $y_i\in B_r(\ox)$, $\nabla e_{\lambda}^{\ox^*}\varphi (x_i)+\ox^* \in B_r(\ox^*)$, and
\begin{equation*}
\varphi (y_i) = \varphi\big(P_{\lambda}^{\ox^*}\varphi (x_i)\big)= \psi_{\lambda} (x_i) <\varphi (\ox) + \varepsilon,\quad i=1,2,
\end{equation*}
which leads us to $(y_i,\nabla e^{\ox^*}_{\lambda}\varphi (x_i)+\ox^*)\in W^{\varepsilon} \cap \gph \partial \varphi$. Observing by \eqref{tilt22} that $x_i = y_i+\lambda J^{-1}(\nabla e^{\ox^*}_{\lambda}\varphi (x_i))$, we get by the local monotonicity of $\partial \varphi$ with respect to $W^{\varepsilon}$ and the global monotonicity of $J^{-1}$ on $X^*$ that
\begin{eqnarray*}
\begin{array}{ll}
\big\langle \nabla e^{\ox^*}_{\lambda}\varphi (x_1)-\nabla e^{\ox^*}_{\lambda}\varphi (x_2),x_1-x_2\big\rangle =\big\langle  \nabla e^{\ox^*}_{\lambda}\varphi (x_1)-\nabla e^{\ox^*}_{\lambda}\varphi (x_2),y_1-y_2\big\rangle \\\\
+ \lambda\big\langle\nabla e^{\ox^*}_{\lambda}\varphi(x_1)-\nabla e^{\ox^*}_{\lambda}\varphi (x_2),J^{-1}\big(\nabla e^{\ox^*}_{\lambda}\varphi (x_1)\big)-J^{-1}\big(\nabla e^{\ox^*}_{\lambda}\varphi (x_2)\big)\big\rangle\ge 0.
\end{array}
\end{eqnarray*}
This tells us that $\nabla e^{\ox^*}_{\lambda}\varphi$ is locally monotone around $\ox$. Utilizing finally the result of \cite[Theorem~4.1.4]{Hiriart-Convex} which proof works also for a general Banach space, we arrive at the local convexity of $e^{\ox^*}_{\lambda}\varphi$ around $\ox$, and thus verify {\bf(iv)}.  

Let us next proceed with the proof of implication {\bf(iv)}$\Longrightarrow${\bf(i)}. Assume that the $\ox^*$-Moreau envelope $e_{\lambda}^{\ox^*} \varphi$ is locally convex around $\ox$ for all $\lm>0$ sufficiently small, that $X$ is smooth and $2$-uniformly convex, and that $\varphi$ is prox-bounded, prox-regular, and weakly sequentially l.s.c.\ around $\ox$. Fixing $\lambda\in(0,\lambda_0)$, suppose without loss of generality that $e_{\lambda}^{\ox^*}\varphi$ is convex on $U_\lambda$. Utilizing the first-order characterization of ${\cal C}^1$-smooth convex functions in \cite[Theorem~4.1.1]{Hiriart-Convex}, which proof holds in  general Banach spaces, gives us the inequality
\begin{equation*}
e_{\lambda}^{\ox^*}\varphi(x)\geq e_{\lambda}^{\ox^*}\varphi(u)+\langle\nabla e_{\lambda}^{\ox^*}\varphi(u),x-u\rangle\;\mbox{ for all }\;x,u\in U_\lambda. 
\end{equation*} 
Select convex neighborhoods $U\subset U_{\lambda}$ of $\ox$ and $V$ of $\ox^*$ such that $U\subset B_\gamma(\ox)$, $V\subset B_\gamma (\ox^*)$, and 
\begin{equation*}
x+\lambda J^{-1}(x^*-\ox^*) \in U_{\lambda}\;\text{ whenever }\;(x,x^*)\in U\times V.
\end{equation*}
Theorem~\ref{main:VC} reduces justifying the variational convexity of $\varphi$ at $\ox$ for $\ox^*$ to verifying that 
\begin{equation}\label{convexineq}
\varphi(x^\prime)\geq \varphi(x)+\langle x^*,x^\prime -x \rangle\;\mbox{ for all }\;x^\prime  \in U,\;(x,x^*) \in (U_\gamma\times V)\cap \gph\partial\varphi,
\end{equation}
where $U_\gamma =\{x\in U\mid \varphi (x) < \varphi (\ox) + \gamma\}$. Pick $x^\prime \in U$, $(x,x^*) \in(U_\gamma\times V)\cap \gph\partial\varphi$ and deduce from \eqref{att} that $(x,x^*) \in\gph T_{\gamma}^{\varphi}$. Employing \eqref{tilt1} and \eqref{tilt2} gives us the equalities
\begin{equation*}
x = P^{\ox^*}_{\lambda}\varphi \big( x+\lambda J^{-1}(x^*-\ox^*)\big)\;\mbox{ and }\;\nabla e^{\ox^*}_{\lambda}\varphi\big(x+\lambda J^{-1}(x^*-\ox^*)\big)=x^* -\ox^*,
\end{equation*}
which yield in turn to the relationships
\begin{eqnarray*}
&&\varphi (x^\prime) - \langle  \ox^*,x^\prime \rangle  + \frac{1}{2\lambda}\big\|x'-\big(x'+\lambda J^{-1}(x^*-\ox^*)\big)\big\|^2 \\
&\ge& e^{\ox^*}_{\lambda}\varphi\big(x'+\lambda J^{-1}(x^*-\ox^*)\big) \\
&\ge& e^{\ox^*}_{\lambda}\varphi\big(x+\lambda J^{-1}(x^*-\ox^*)\big) + \big\langle \nabla e^{\ox^*}_{\lambda}\varphi\big(x+\lambda J^{-1}(x^*-\ox^*)\big), x'-x\big\rangle \\
&=& \left[\varphi (x) -\langle \ox^*,x\rangle  + \frac{1}{2\lambda}\big\|x-\big(x+\lambda J^{-1}(x^*-\ox^*)\big)\big\|^2\right] + \langle x^*-\ox^*,x^\prime  -x\rangle.
\end{eqnarray*}
After the simplification of the above, we arrive at the estimate
\begin{equation*}
\varphi(x^\prime)\ge\varphi(x) + \langle x^*, x^\prime  - x\rangle ,
\end{equation*}
which verifies \eqref{convexineq} and thus completes the proof of the theorem.
\end{proof}

If $\ox$ is a {\em stationary point} of $\ph$, i.e., $\ox^*=0$, 
we get the immediate consequence of Theorem~\ref{1stequi}: 

\begin{Corollary} Let $X$ be a $2$-uniformly convex Banach space which norm is G\^ateaux differentiable at nonzero points, and let $\varphi:X\to\overline{\R}$ be an l.s.c.\ and prox-bounded function with $\ox\in\dom\varphi$ and $0\in\partial\varphi(\ox)$. Consider the following assertions: 
\begin{itemize}
\item[\bf(i)] $\varphi$ is variationally convex at $\ox$ for $0$.
\item[\bf (ii)] $\varphi$ is prox-regular at $\ox$ for $0$, and the Moreau envelope $e_{\lambda} \varphi$ is locally convex around $\ox$ for all small $\lambda>0$. 
\end{itemize}
Then we have{ implication \bf(i)}$\Longrightarrow${\bf(ii)} under the general assumptions made, while the reverse implication {\bf(ii)}$\Longrightarrow${\bf(i)} holds provided that $\varphi$ is weakly sequentially l.s.c.\ around $\ox$.
\end{Corollary} 

The last result of this section establishes a characterization of variational convexity for functions defined on {\em Hilbert spaces}, where we can replace the $\ox^*$-Moreau envelope \eqref{defi:Motilted} around $\ox$ by its standard counterpart \eqref{el} around $\ox+\lm\ov$, where $\ox^*\in\partial\ph(\ox)\subset X^*$ is replaced by $\ov\in\partial\ph(\ox)\subset X$ due to $X=X^*$. The corollary below is a Hilbert space extension of the finite-dimensional characterization obtained in \cite[Theorem~3.2]{kmp22convex}.

\begin{Corollary}  
Let $X$ be a Hilbert space, and let $\varphi: X\to\overline{\R}$ be an l.s.c.\ and prox-bounded function with $\bar{x}\in\dom\varphi$ and $\bar{v}\in\partial\varphi(\bar{x})$. Consider the following assertions:
\begin{itemize}
\item[\bf(i)] $\varphi$ is variationally convex at $\ox$ for $\ov$.
\item[\bf(ii)] $\varphi$ is prox-regular at $\ox$ for $\ov$, and the Moreau envelope $e_\lambda\varphi$ is locally convex around $\bar{x}+\lambda\bar{v}$ for all small $\lambda>0$.
\end{itemize}    
Then implication {\bf(i)}$\Longrightarrow${\bf(ii)} is satisfied under the general assumptions imposed, while the reverse implication {\bf(ii)}$\Longrightarrow${\bf(i)} holds if in addition $\varphi$ is weakly sequentially l.s.c.\ around $\ox$.
\end{Corollary} 
\begin{proof} To deduce this corollary from  Theorem~\ref{1stequi}, it suffices to show that the local convexity of the $\bar{v}$-Moreau envelope $e^{\bar{v}}_{\lambda}\varphi$ around $\ox$ is equivalent to the local convexity of the standard one $e_{\lambda}\varphi$ around $\ox+\lambda\bar{v}$. Indeed, we get from \eqref{rela:Moreau} that
\begin{equation}\label{4.55}
e_{\lambda}^{\bar{v}}\varphi (x) = e_{\lambda}\varphi (x+\lambda \bar{v}) -\langle \bar{v},x\rangle -\frac{\lambda}{2}\|\bar{v}\|^2\;\text{ for all }\;x\in X.
\end{equation}
Defining $\psi_{\bar{v}} (x): = x + \lambda \bar{v}$ for $x\in X$, the equality in \eqref{4.55} can be rewritten as 
\begin{equation*}
e_{\lambda}^{\bar{v}}\varphi (x) = (e_{\lambda}\varphi \circ \psi_{\bar{v}})(x)  -\langle \bar{v},x\rangle -\frac{\lambda}{2}\|\bar{v}\|^2,\quad x\in X.
\end{equation*}
Since $\psi_{\bar{v}}$ is an affine mapping while $\la \bar{v},\cdot\ra$ is a convex function, the local convexity of $e_{\lambda}\varphi$ around $\ox + \lambda \ov$ yields the same property of $e^{\bar{v}}_{\lambda} \varphi$ around $\ox$. On the other hand, \eqref{4.55} is represented as
\begin{equation*}
e_{\lambda} \varphi (x) = e^{\bar{v}}_{\lambda}\varphi (x-\lambda \bar{v}) + \la \bar{v},x\ra -\dfrac{\lambda}{2}\|\bar{v}\|^2 = (e^{\bar{v}}_{\lambda}\varphi \circ \theta_{\bar{v}})(x) + \la \bar{v},x\ra -\dfrac{\lambda}{2}\|\bar{v}\|^2,
\end{equation*}
where $\theta_{\bar{v}}(x):=x-\lambda \bar{v}$, $x\in X$, is an affine mapping. Thus the local convexity of $e^{\bar{v}}_{\lambda} \varphi$ around $\ox$ reduces to that of $e_{\lambda}\varphi$ around $\ox + \lambda \bar{v}$. Combining finally Theorem~\ref{1stequi} with the above equivalence, we arrive at the claimed result.
\end{proof}

\end{document}